\documentclass[a4paper,12pt,reqno]{amsart}
\pdfoutput=1
\usepackage{graphicx}
\usepackage{a4wide}
\usepackage{ifthen}
\usepackage{url}
\usepackage{hyperref}

\newcommand{\parentheses}[4][]%
{\ifthenelse{\equal{#1}{}}{\left#2}{\csname#1\endcsname#2}%
#4\ifthenelse{\equal{#1}{}}{\right#3}{\csname#1\endcsname#3}}

\newcommand{\ceiling}[2][]{\parentheses[#1]{\lceil}{\rceil}{#2}}

\let\Im=\relax
\DeclareMathOperator{\Im}{Im}

\let\Re=\relax
\DeclareMathOperator{\Re}{Re}

\renewcommand{\MR}[1]{}

\newtheorem{theorem}{Theorem}
\newtheorem{lemma}{Lemma}[section]

\DeclareMathOperator{\R}{\mathbb{R}}
\DeclareMathOperator{\Z}{\mathbb{Z}}

\makeatletter
\@ifundefined{subjclassname@2010}{%
  \@namedef{subjclassname@2010}{%
    \textup{2010} Mathematics Subject Classification}
}
\makeatother

\allowdisplaybreaks

\title[$p$-adic valuation of the number of alternating sign
matrices]{A precise description of the $\boldsymbol{p}$-adic valuation of the
  number of alternating sign matrices}

\author{Clemens Heuberger}
\address{Institut f\"ur Mathematik B\\Technische Universit\"at Graz\\Austria}
\email{clemens.heuberger@tugraz.at}
\thanks{This paper was written while C.~Heuberger 
  was a visitor at the
  Center of Experimental Mathematics at the University of Stellenbosch and
  while he was a visitor at the Institute of Mathematics at the University of
  Debrecen supported by the Action Austria-Hungary (No. 75\"ou1). He
  thanks both institutions
 for their hospitality. He is also supported
by the Austrian
Science Foundation FWF, project S9606, that is part of the
Austrian National Research Network ``Analytic Combinatorics
and Probabilistic Number Theory.''}
\author{Helmut Prodinger}
\address{Department of Mathematics\\University of Stellenbosch\\South Africa}
\email{hproding@sun.ac.za}
\thanks{H.~Prodinger is supported by the NRF grant 2053748 
of the South African National Research Foundation and by
the Center of Experimental Mathematics of the University of Stellenbosch.}



\begin{document}

\begin{abstract} Following Sun and Moll~\cite{Sun-Moll:2009:p-adic}, we study $v_p(T(N))$, the $p$-adic valuation of the counting function of the alternating sign matrices. We find an exact analytic expression for it that exhibits
the fluctuating behaviour, by means of Fourier coefficients. The method
is the Mellin-Perron technique, which is familiar in the analysis of 
the sum-of-digits function and related quantities.
\end{abstract}

\maketitle

% \ifprivate \thispagestyle{headings}\pagestyle{headings} \markboth{\jobname{}
% rev. \SVNRevision{} ---
%   \SVNDate{} \SVNTime}{\jobname{} rev. \SVNRevision{} --- \SVNDate{} \SVNTime}
% \fi

\section{Introduction}

Sun and Moll~\cite{Sun-Moll:2009:p-adic} consider the counting function
\begin{equation*}
	T(N)=\prod_{j=0}^{N-1}\frac{(3j+1)!}{(N+j)!},
\end{equation*}
which is famous because of the enumeration of the \emph{Alternating Sign Matrices.} Their paper provides also some historic remarks about the fascinating story of this subject.

Sun and Moll write:
\begin{quotation}\slshape
Given an interesting sequence of integers, it is a natural question to explore the
structure of their factorization into primes. This is measured by the $p$-adic valuation
of the elements of the sequence.
\end{quotation}

Indeed, in~\cite{Sun-Moll:2009:p-adic} the $p$-adic valuation $v_p(T(N))$ is studied, i.e., $v_p(m)$ denotes the maximum $k$ such that $p^k$ divides $m$. 

Most results are, however, for $p=2$ only. 

Their key result is (Corollary 2.2):
\begin{equation*}
v_p(T(N))=\frac1{p-1}\biggl(\sum_{j=0}^{N-1}S_p(N+j)-\sum_{j=0}^{N-1}S_p(3j+1)\biggr),
\end{equation*}
where $S_p(k)$ is the sum-of-digits function of $k$ to the base $p$.

We use an analytic approach (the Mellin-Perron technique and its extension, \cite{Flajolet-Grabner-Kirschenhofer-Prodinger:1994:mellin, Grabner-Kirschenhofer-Prodinger:1998})
to derive an \emph{exact} analytic expression for this function, and this works for \emph{all} primes. The periodicities are made fully explicit by computing the relevant Fourier coefficients. 

\section{Results}

Now we study the function $v_p(T(N))$ more closely. The prime $p$ is fixed
throughout this paper, everything may depend on it, including implicit
constants in the $O$-notation.

We will prove that asymptotically,
\begin{equation}\label{eq:result-summary}
  v_p(T(N))=N \log_p\frac2{\sqrt3}+N\Phi(\log_p N)+O(\sqrt N),
\end{equation}
where
\begin{equation*}
  \Phi(x)=\sum_{k\in\Z\setminus\{0\}}c_k e^{2k \pi i x}
\end{equation*}
is a $1$-periodic and continuous function of mean $0$ with Fourier coefficients
\begin{equation*}
  c_k=\frac{(1-2^{1+\chi_k}+3^{\chi_k})\zeta(\chi_k)}{\chi_k(1+\chi_k)\log
    p}\qquad\text{with}\qquad
  \chi_k=\frac{2k\pi i}{\log p}
\end{equation*}
for $k\in\Z\setminus\{0\}$. 

The $O$-term in \eqref{eq:result-summary} can also be computed explicitly. Here, the cases
$p\equiv 1\pmod 3$, $p\equiv-1\pmod 3$ and $p=3$ differ in several aspects,
thus we formulate one theorem for each
case: Theorems~\ref{theorem:result-case-1}, \ref{theorem:result-case-2} and
\ref{theorem:p-3}, respectively.

The occurring quantities are expressed in terms of the Hurwitz $\zeta$ function defined by
\begin{equation*}
  \zeta(s,\alpha)=\sum_{n>-\alpha}\frac1{(n+\alpha)^s},
\end{equation*}
where we allow arbitrary $\alpha\in\R$. For $0<\alpha\le 1$, this coincides
with the usual definition. In our version, however, it is $1$-periodic in the
second variable, which will be useful for our purposes. The special case
$\alpha\in\Z$ corresponds to the Riemann $\zeta$ function $\zeta(s)$.

\begin{theorem}\label{theorem:result-case-1} Assume that $p\equiv 1\pmod 3$.
  Then the asymptotic expansion (which is also exact) of the function $v_p(T(N))$ is given by
  \begin{equation*}
    v_p(T(N))=N \log_p\frac2{\sqrt3}+N\Phi(\log_p N)+\Psi(N)
    +\frac19\log_p N+f_0(N).
  \end{equation*}
  Here, 
  \begin{itemize}
  \item the function
    \begin{equation*}
    \Phi(x)=\sum_{k\in\Z\setminus\{0\}}c_k e^{2k \pi i x}
  \end{equation*}
  is $1$-periodic and continuous;
\item the function 
  \begin{align*}
    \Psi(N)&=\Bigl(N+\frac13\Bigr)\psi_{1}\Bigl(\log_p\Bigl(N+\frac
    13\Bigr)\Bigr)+
    \Bigl(N-\frac13\Bigr)\psi_{-1}\Bigl(\log_p\Bigl(N-\frac
    13\Bigr)\Bigr)\\&\qquad-
    N\psi_{0}(\log_p N),
  \end{align*}
  is expressed in terms of the continuous $1$-periodic functions
  \begin{equation*}
    \psi_{j}(x)=\sum_{k\in\Z\setminus\{0\}}d_{k,j}e^{2k\pi i x},\qquad j\in\{-1,0,1\}
  \end{equation*}
  and satisfies
  \begin{equation*}
    \Psi(N)=O(N^{1/2}\log N);
  \end{equation*}
  
\item the least order term is given by
  \begin{equation*}
     f_0(N) =\frac{1}{3} \log _p\biggl(\frac{\Gamma \bigl(\frac{1}{3}\bigr)}{\Gamma
   \bigl(\frac{2}{3}\bigr)}\biggr)
+\frac{g_1(N)-g_{-1}(N)}{6}-\frac{1}{9 \log p}+\frac{p+1}{6 (p-1)}
  \end{equation*}
  with
  \begin{equation}\label{eq:g-j-Definition}
    g_j(N)=\Bigl(1+\frac{j}{3N}\Bigr)N\log_p\Bigl(1+\frac{j}{3N}\Bigr)
    \underset{N\to\infty}{\sim}\frac{j}{3 \log p}+O\Bigl(\frac{1}{N}\Bigr),
  \end{equation}
  thus $f_0(N)$ converges to
  \begin{equation*}
    \frac{1}{3} \log _p\biggl(\frac{\Gamma \bigl(\frac{1}{3}\bigr)}{\Gamma
      \bigl(\frac{2}{3}\bigr)}\biggr)
    +\frac{p+1}{6 (p-1)}
  \end{equation*}
  for $N\to\infty$.
  \end{itemize}
 
  The Fourier coefficients of $\Phi(x)$, $\psi_{\pm1}(x)$ and $\psi_{0}$
  are given by
  \begin{align*}
    c_k&=\frac{(1-2^{1+\chi_k}+3^{\chi_k})\zeta(\chi_k)}{\chi_k(1+\chi_k)\log p}&&\text{with}&
    \chi_k&=\frac{2k\pi i}{\log p},\\
    d_{k,j}&=\frac{\zeta(\chi_k,j/3)}{\chi_k(1+\chi_k)\log p}&&\text{for}&j&\in\{\pm1\},\\
    d_{k,0}&=d_{k,1}+d_{k,-1}=\frac{(3^{\chi_k}-1)\zeta(\chi_k)}{\chi_k(1+\chi_k)\log p},
  \end{align*}
  respectively.
\end{theorem}

The case $p\equiv -1\pmod 2$ is similar, but the structure of $\Psi(x)$ is
more complicated: it is now composed of $2$-periodic functions instead of
$1$-periodic functions. On the other hand, the least order term is less
complicated.

\begin{theorem}\label{theorem:result-case-2}
  Assume that $p\equiv -1\pmod 3$.
  Then the asymptotic expansion (which is also exact) of the function $v_p(T(N))$ is given by
  \begin{equation*}
     v_p(T(N))=N \log_p\frac2{\sqrt3}+N\Phi(\log_p N)+\Psi(N)+\frac{p+1}{6 (p-1)}.
  \end{equation*}
  Here, 
  \begin{itemize}
  \item the function
    \begin{equation*}
    \Phi(x)=\sum_{k\in\Z\setminus\{0\}}c_k e^{2k \pi i x}
  \end{equation*}
  is $1$-periodic and continuous;
\item the function 
  \begin{align*}
    \Psi(N)&=\Bigl(N+\frac13\Bigr)\psi_{1}\Bigl(\log_p\Bigl(N+\frac
    13\Bigr)\Bigr)+
    \Bigl(N-\frac13\Bigr)\psi_{-1}\Bigl(\log_p\Bigl(N-\frac
    13\Bigr)\Bigr)\\&\qquad-
    N\psi_{0}(\log_p N),
  \end{align*}
  is expressed in terms of the continuous $2$-periodic functions
  \begin{equation*}
    \psi_{j}(x)=\sum_{k\in\Z\setminus\{0\}}d_{k,j}e^{k\pi i x},\qquad j\in\{-1,0,1\}
  \end{equation*}
  and satisfies
  \begin{equation*}
    \Psi(N)=O(N^{1/2}\log N).
  \end{equation*}
  
  \end{itemize}

  The Fourier coefficients of $\Phi(x)$, $\psi_{\pm1}(x)$ and $\psi_{0}$
  are given by
  \begin{align*}
    c_k&=\frac{(1-2^{1+\chi_k}+3^{\chi_k})\zeta(\chi_k)}{\chi_k(1+\chi_k)\log p}&&\text{with}&
    \chi_k&=\frac{2k\pi i}{\log p},\\
    d_{k,j}&=\frac{\zeta(\chi_{k/2},j/3)+(-1)^k\zeta(\chi_{k/2},-j/3)}{2\chi_{k/2}(1+\chi_{k/2})\log p}&&\text{for}&j&\in\{\pm1\},\\
    d_{k,0}&=d_{k,1}+d_{k,-1}=
    \begin{cases}
      \frac{(3^{\chi_{k/2}}-1)\zeta(\chi_{k/2})}{\chi_{k/2}(1+\chi_{k/2})\log
        p},&\text{ if $k$ is even},\\
      0,&\text{ if $k$ is odd},
    \end{cases}
  \end{align*}
  respectively.
\end{theorem}

Finally, for $p=3$, the result is that the $O$-term in
\eqref{eq:result-summary} can be omitted:

\begin{theorem}\label{theorem:p-3}
  Then $v_3(T(N))$ is given by
  \begin{equation*}
    v_3(T(N))=N\Bigl(\log_3 2-\frac12\Bigr)+N\Phi(\log_3 N),
  \end{equation*}
  where 
  \begin{equation*}
    \Phi(x)=\sum_{k\in\Z\setminus\{0\}}c_ke^{2k\pi i x}
  \end{equation*}
  is a $1$-periodic continuous function with Fourier coefficients
  \begin{equation*}
    c_k=\frac{2(1-2^{\chi_k})\zeta(\chi_k)}{\chi_k(1+\chi_k)\log p}\qquad\text{with}\qquad
    \chi_k=\frac{2k\pi i}{\log p}.
  \end{equation*}
\end{theorem}

As examples, Figures~\ref{fig:picture-p-2}, \ref{fig:picture-p-3}  and \ref{fig:picture-p-7} show
$v_p(T(N))/N$ (gray dots) and an approximation of
$\log_p\frac2{\sqrt3}+\Phi(\log_p N)$ (black line) for $p=2$, $p=3$ and $p=7$,
respectively. For the approximation, $400$ Fourier coefficients of $\Phi$ are
used. Note that for $p=3$, the approximation is much better as the only
approximation consisted in truncating the Fourier series, in contrast to the
lower order terms present in the other cases.

\begin{figure}
  \centering
  \includegraphics{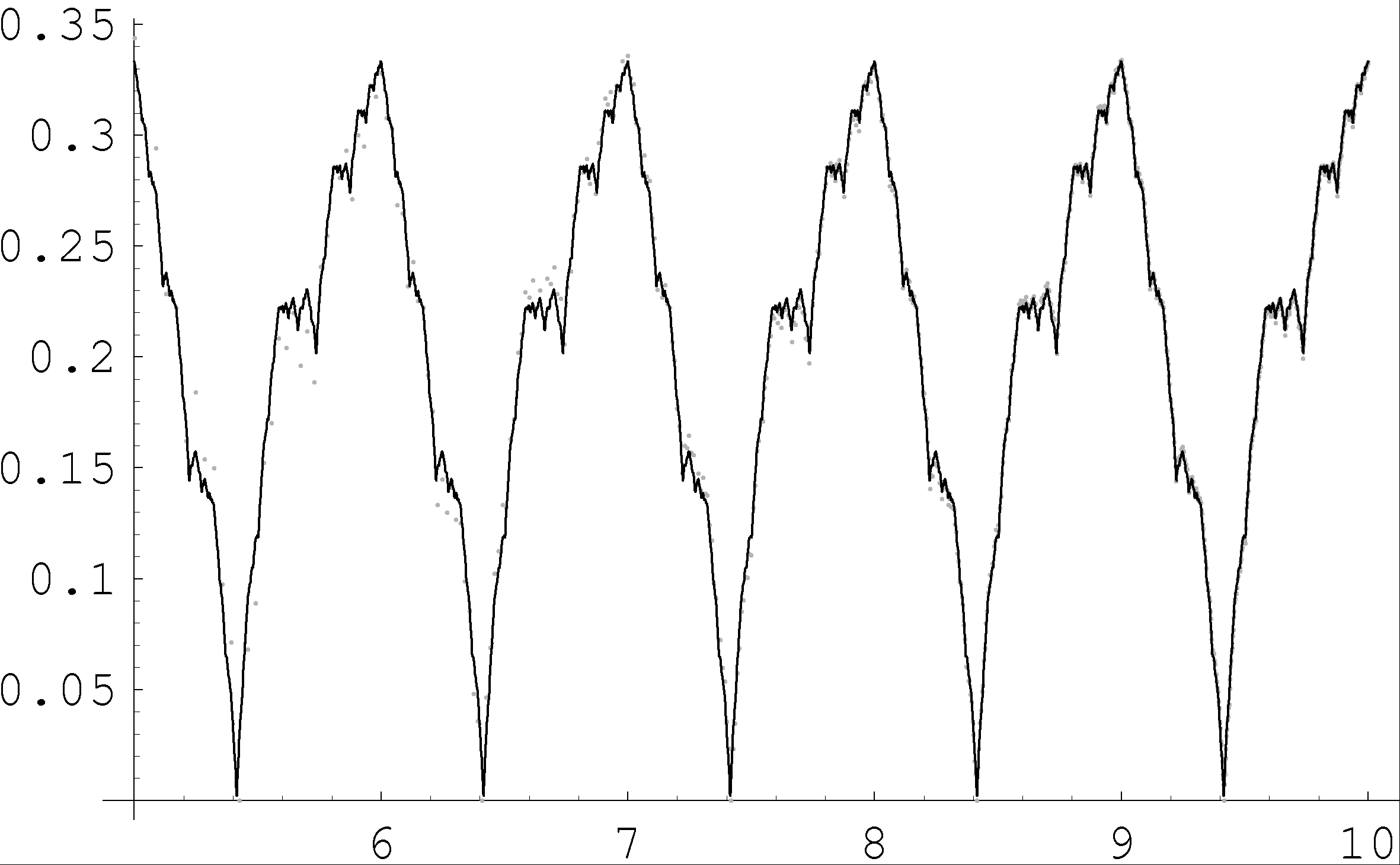}
  \caption{Periodic fluctuation for $p=2$}
  \label{fig:picture-p-2}
\end{figure}

\begin{figure}
  \centering
  \includegraphics{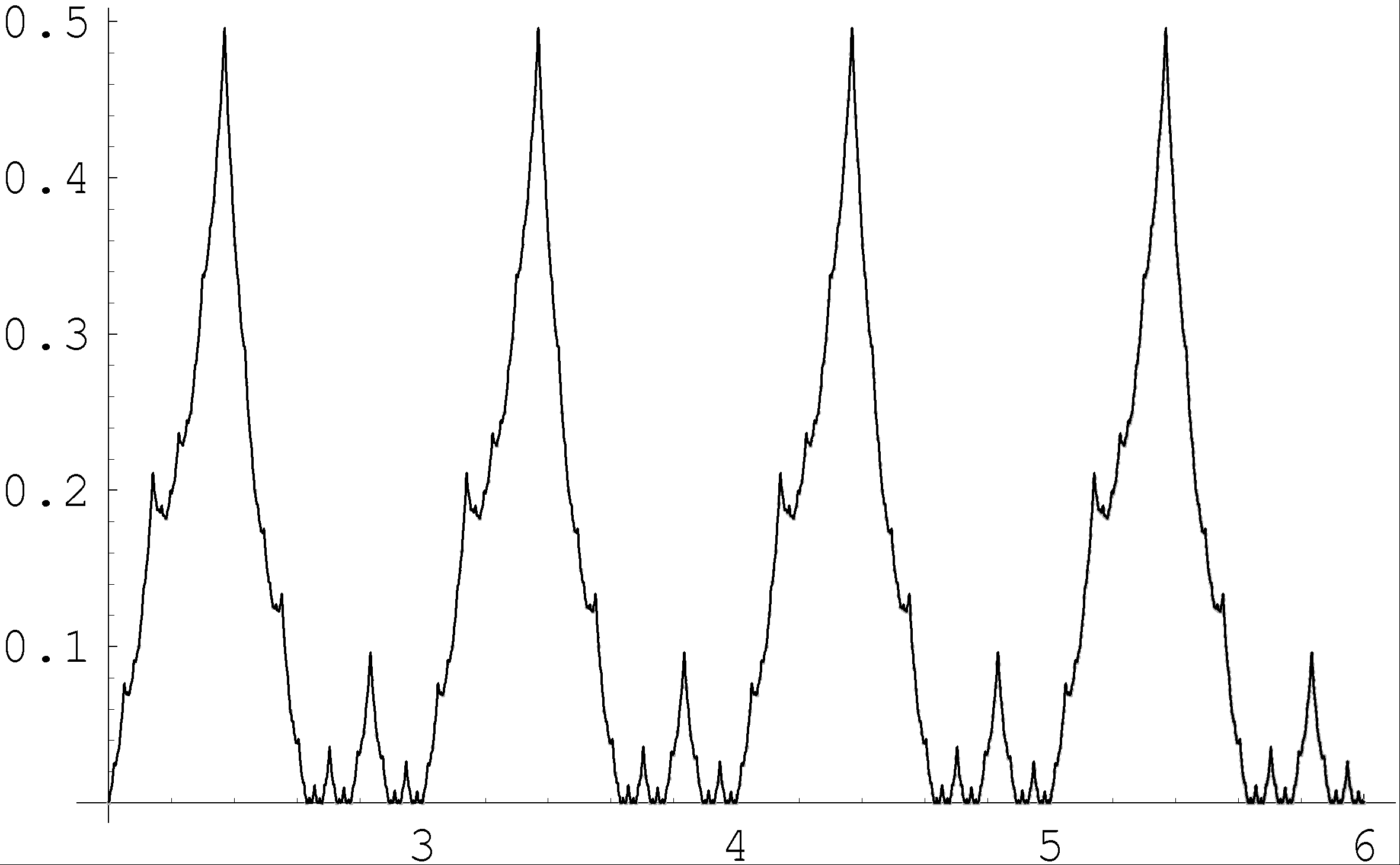}
  \caption{Periodic fluctuation for $p=3$}
  \label{fig:picture-p-3}
\end{figure}

\begin{figure}
  \centering
  \includegraphics{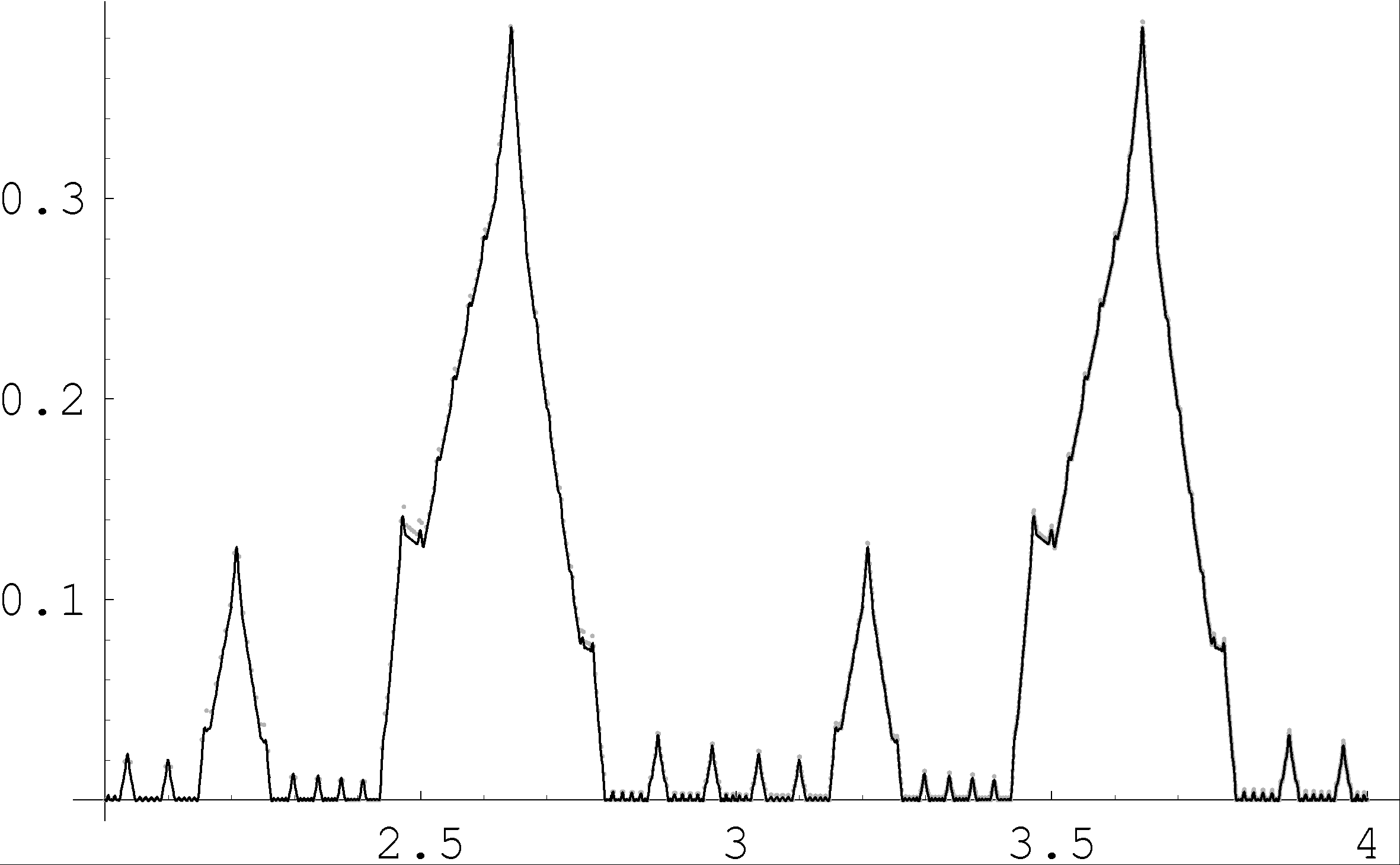}
  \caption{Periodic fluctuation for $p=7$}
  \label{fig:picture-p-7}
\end{figure}

Our analysis builds on a result of Delange~\cite{Delange:1975:chiffres}:

\begin{lemma}[Delange~\cite{Delange:1975:chiffres}]\label{lemma:Delange}
  \begin{equation*}
    \sum_{n=0}^{N-1}S_p(n)=\frac{p-1}{2}N\log_pN+Nc_0^{(1)}+N\Phi^{(1)}(\log_p N),
  \end{equation*}
  with
  \begin{equation*}
    c_0^{(1)}=\frac{p-1}{2\log p}(\log 2\pi-1)-\frac{p+1}{4},
  \end{equation*}
  a $1$-periodic continuous function of mean zero
  \begin{equation*}
    \Phi^{(1)}(x)= \sum_{k\neq0}c_k^{(1)}e^{2k\pi ix},
  \end{equation*}
  and
  \begin{equation*}
    c_k^{(1)}=-\frac{p-1}{\log p}\frac1{\chi_k(1+\chi_k)}\zeta(\chi_k),\quad \chi_k=\frac{2k\pi i}{\log p}
  \end{equation*}
  for $k\neq 0$.
\end{lemma}

This already leads to the simplification
\begin{multline}\label{simplified}
v_p(T(N))=\frac{1}{2}N\log_p N+N\Bigl(\log_p2+\frac{c_0^{(1)}}{p-1}\Bigr)\\
+\frac1{p-1}2N\Phi^{(1)}(\log_p2N)
-\frac1{p-1}N\Phi^{(1)}(\log_p N)+\frac1{1-p}\sum_{j=0}^{N-1}S_p(3j+1).
\end{multline}

In the rest of this paper we will analyse the sum
\begin{align*}
\frac1{1-p}\sum_{j=0}^{N-1}S_p(3j+1).
\end{align*}

For $p=3$, we have $S_3(3j+1)=1+S_3(j)$, so that another application of
Lemma~\ref{lemma:Delange} directly yields the result described in Theorem~\ref{theorem:p-3}.

In the following, we will therefore assume that $p\neq 3$.  
However, we will not follow Delange's original (elementary) approach,
but rather use the Mellin-Perron approach as described in \cite{Flajolet-Grabner-Kirschenhofer-Prodinger:1994:mellin} and
extended in \cite{Grabner-Kirschenhofer-Prodinger:1998}.

\section{Rewriting the Remaining Sum as an Integral}

The aim of this section is to prove the following lemma.

\begin{lemma}\label{lemma:Mellin-Perron-result}
  We have 
  \begin{equation}\label{eq:Mellin-Perron-result}
    \frac1{1-p}\sum_{n=0}^{N-1} S_p(3n+1)=-\frac1{1-p}\Bigl(\frac{N}{2}-\frac{3 N^2}{2}\Bigr)+I
  \end{equation}
  with
  \begin{equation*}
    I=\frac1{2\pi i}\sum_{j=-1}^1\int_{2-i\infty}^{2+i\infty}\Lambda_j(s)\Bigl(N+\frac
  j3\Bigr)^{s+1}\frac{ds}{s(s+1)}
  \end{equation*}
  and
  \begin{equation}\label{eq:lambda-expression}
    \Lambda_j(s)=\frac{\zeta\bigl(s, \tfrac j3\bigr)+p^s\zeta\bigl(s, \tfrac{uj}3\bigr)}{p^{2s}-1},
  \end{equation}
  where $u\in\{\pm 1\}$  is chosen such that $p\equiv u\pmod 3$.

\end{lemma}

\begin{proof}
We intend to use the Mellin-Perron summation formula in the version 
\begin{equation}\label{eq:Mellin-Perron}
  \sum_{\substack{n\\-\alpha<n<N}}(N-n)a(n)=\frac{1}{2\pi i}\int_{C-i\infty}^{C+i\infty} \Biggl(\sum_{\substack{n\\n>-\alpha}}\frac{a(n)}{(n+\alpha)^s}\Biggr)(N+\alpha)^{s+1}\frac{ds}{s(s+1)},
\end{equation}
where $\alpha\in\R$ and $C$ is in the half-plane of absolute convergence of the
Dirichlet series $\sum_{n>-\alpha} \frac{a(n)}{(n+\alpha)^{s}}$,
cf.~\cite[(4.3)]{Grabner-Kirschenhofer-Prodinger:1998}.
Note that we do not impose the frequently used restriction $0<\alpha\le 1$, but
we sum over $n>-\alpha$, which amounts to the same. This version, however,
smoothes the following calculations. 

For $j\in\{-1,0,1\}$, we set
\begin{equation*}
  a_j(n)=\frac{S_p(3n+j)-S_p(3n+j-1)-1}{1-p}
\end{equation*}
and 
\begin{equation*}
  \Lambda_j(s)=\sum_{n>-j/3} \frac{a_j(n)}{\left(n+\frac j3\right)^s}.
\end{equation*}
Note that $\Lambda_j(s)$ converges absolutely for $\Re s>1$. The additional
summand $-1$ in the numerator of $a_j(n)$ leads to simpler expressions for
$\Lambda_j(s)$, the denominator also simplifies the expressions and is also
present in \eqref{simplified}.

To prove \eqref{eq:Mellin-Perron-result}, we
apply \eqref{eq:Mellin-Perron} with $a(n)=a_j(n)$ and $\alpha=j/3$ for
$j\in\{-1,0,1\}$, sum up the result and use Abel summation. This yields
\begin{align*}
   I&=\sum_{n=1}^{N-1}(N-n)(a_1(n)+a_0(n)+a_{-1}(n))+ N a_1(0)\\
   &=\frac1{1-p}\sum_{n=1}^{N-1}(N-n)(S_p(3n+1)-S_p(3n-2)-3)+\frac1{1-p}N(S_p(1)-S_p(0)-1)\\
   &=\frac1{1-p}\biggl(S_p(3N-2)+\sum_{n=1}^{N-2} S_p(3n+1)
   (N-n-(N-(n+1))) \\
   &\qquad- (N-1)S_p(1)+ N S_p(1)+\frac{N}{2}-\frac{3 N^2}{2}\biggr)\\
   &=\frac1{1-p}\biggl(S_p(3N-2)+\sum_{n=1}^{N-2} S_p(3n+1)+ S_p(1)+\frac{N}{2}-\frac{3 N^2}{2}\biggr)\\
   &=\frac1{1-p}\sum_{n=0}^{N-1} S_p(3n+1)+\frac1{1-p}\Bigl(\frac{N}{2}-\frac{3 N^2}{2}\Bigr),
\end{align*}
as requested.

In order to compute $I$, we need explicit expressions for the $\Lambda_j$. 
It is well-known (and easy to see) that for a positive integer $m$, we have
\begin{equation*}
  S_p(m)-S_p(m-1)=1 - (p-1)v_p(m),
\end{equation*}
which immediately results in
\begin{equation*}
  a_j(n)=v_p(3n+j).
\end{equation*}

Thus we get
\begin{equation*}
  \Lambda_j(s)=\sum_{n>-j/3}\frac{v_p(3n+j)}{\left(n+\frac j3\right)^s}=\sum_{k\ge 1}\sum_{\substack{n>-j/3\\
      v_p(3n+j)\ge k}}\frac1{\left(n+\frac j3\right)^s}
  =\sum_{k\ge 1}\sum_{\substack{n>-j/3\\
      p^k\mid 3n+j}}\frac1{\left(n+\frac j3\right)^s}.  
\end{equation*}
By definition of $u$, we have $\frac{3n+j}{p^k}\equiv ju^k\pmod{3}$. Thus
$3n+j$ is divisible by $p^k$ if and only if there is an integer $m$
such that $3n+j=p^k(3m+ju^k)$. This results in
\begin{equation*}
  \Lambda_j(s)=\sum_{k\ge 1}\sum_{m>-ju^k/3}\frac1{p^{ks}\bigl(m+\frac {ju^k}3\bigr)^s}.
\end{equation*}
We now split the sum over $k$ according to $k$'s parity and get
\begin{equation}\label{eq:Lambda-computed}
\begin{aligned}
  \Lambda_j(s)&=\sum_{k\ge
    1}\sum_{m>-j/3}\frac1{p^{2ks}\left(m+\frac {j}3\right)^s}
  +\sum_{k\ge 1}\sum_{m>-uj/3}\frac1{p^{(2k-1)s}\left(m+\frac
    {uj}3\right)^s}\\
&=\frac{\zeta\bigl(s, \tfrac j3\bigr)+p^s\zeta\bigl(s, \tfrac{uj}3\bigr)}{p^{2s}-1},
\end{aligned}
\end{equation}
as claimed in \eqref{eq:lambda-expression}.
\end{proof}

\section{Computing the Asymptotic Main Terms}

The further strategy is now to shift the line of integration to the left.
The residues at the poles yield the main terms in the asymptotic expansion.
We prove the following lemma.

\begin{lemma}\label{lemma:residues} We have
  \begin{equation}\label{eq:residues}
  \begin{aligned}
  I&=-\frac{3 N^2}{2 (1-p)}-\frac{1}{2}N \log _pN+ N \Bigl(-\frac{1}{2} \log _p6 \pi +\frac{1}{2 \log
      p}+\frac{1}{4}\Bigr)\\
  &\qquad +N \Phi^{(2)}(\log_p N ) +\Psi(N)+f_{1}^{(2)}\log_p N
    +f_0^{(2)}\\
  &\qquad +\frac1{2\pi i}\sum_{j=-1}^1\int_{-\frac14-i\infty}^{-\frac14+i\infty}\Lambda_j(s)\Bigl(N+\frac
  j3\Bigr)^{s+1}\frac{ds}{s(s+1)},
\end{aligned}
\end{equation}
where 
\begin{itemize}
\item $\Psi(N)$ has been defined in Theorems~\ref{theorem:result-case-1} and
\ref{theorem:result-case-2}, respectively and has the properties given in these
theorems;
\item the quantities $f_{1}^{(2)}$ and $f_0^{(2)}$ depend on $p$ modulo $3$:
\begin{equation}\label{eq:f-quantities-minus}
  f_{1}^{(2)}=0,\qquad f_0^{(2)}=-\frac1{18}-\frac{1}{9 (1-p)}
\end{equation}
if $p\equiv -1\pmod 3$ and
\begin{equation}\label{eq:f-quantities-plus}
\begin{aligned}
  f_{1}^{(2)}&=\frac19,\\
   f_0^{(2)}&=\frac{1}{3} \log _p\Bigl(\frac{\Gamma \bigl(\frac{1}{3}\bigr)}{\Gamma \bigl(\frac{2}{3}\bigr)}\Bigr)
   +\frac{g_1(N)-g_{-1}(N)}{6}-\frac{1}{9 \log p}-\frac{1}{18}-\frac{1}{9 (1-p)}
\end{aligned}
\end{equation}
if $p\equiv 1\pmod 3$ with the convergent function $g_j(N)$ defined in
\eqref{eq:g-j-Definition};
\item the $1$-periodic function 
\begin{equation}\label{eq:Phi-2-definition}
  \Phi^{(2)}(x)=\sum_{k\in\Z\setminus\{0\}}c_k^{(2)}e^{2k\pi i x}
\end{equation}
with
\begin{equation*}
  c_k=\frac{3^{\chi_{k}}\zeta(\chi_{k})}{\chi_{k}(1+\chi_{k})\log
    p}\qquad\text{ for $k\neq 0$}
\end{equation*}
is continuous.
\end{itemize}

\end{lemma}
\begin{proof}
We shift the line of integration to $\Re s=-1/4$,

\begin{equation*}
    I=\frac1{2\pi i}\sum_{j=-1}^1\int_{-\frac14-i\infty}^{-\frac14+i\infty}\Lambda_j(s)\Bigl(N+\frac
  j3\Bigr)^{s+1}\frac{ds}{s(s+1)}+\sum_{s\text{ poles}}r_s,
\end{equation*}
where $r_s$ denotes the residue of the integrand in $s$.
Here, the usual growth estimates for the $\zeta$ function on vertical lines
(cf. \cite[\S~13.51]{Whittaker-Watson:1969})
\begin{equation}\label{eq:growth-estimates-zeta}
  \zeta(s,\alpha)=O(|t|^{\tau(\sigma)}\log|t|) \qquad\text{with}\qquad
  \tau(\sigma)=\begin{cases}
    0&1\le \sigma,\\
    1-\sigma&1/2\le \sigma\le 1,\\
    1/2&0\le \sigma\le 1/2,\\
    1/2-\sigma&\sigma\le 0
    \end{cases}
\end{equation}
have been used, where $\sigma=\Re s$ and $t=\Im s$, as usual. 

The residue at $s=1$ equals
\begin{equation*}
  r_1=-\frac{3 N^2}{2 (1-p)}-\frac{1}{9 (1-p)}.
\end{equation*}
The residue at $s=0$ is calculated as

\begin{equation*}
  r_0=-\frac{1}{2}N \log _pN+ N \Bigl(-\frac{1}{2} \log _p6 \pi +\frac{1}{2 \log
      p}+\frac{1}{4}\Bigr)+f_{1}^{(2)}\log_p N +f_0^{(2)}+\frac{1}{9 (1-p)},
\end{equation*}
with the quantities $f_{1}^{(2)}$ and $f_0^{(2)}$ defined in
\eqref{eq:f-quantities-minus} and \eqref{eq:f-quantities-plus}.

Finally, there is a simple pole of the integrand at $s=\chi_{k/2}=k\pi i/\log p$ for
$k\in\Z\setminus\{0\}$ with residue
\begin{equation*}
  r_{\chi_{k/2}}:=\sum_{j=-1}^1\frac{\zeta(\chi_{k/2},j/3)+(-1)^k\zeta(\chi_{k/2},uj/3)}{2\chi_{k/2}(1+\chi_{k/2})\log
    p}\Bigl(N+\frac j3\Bigr)^{1+\chi_{k/2}}.
\end{equation*}
Approximating $(N+j/3)^{1+\chi_{k/2}}$ by $N^{1+\chi_{k/2}}$ and correcting the
error yields
\begin{multline}\label{eq:split-Fourier-contribution}
  r_{\chi_{k/2}}=\frac{\sum_{j=-1}^1\left(\zeta(\chi_{k/2},j/3)+(-1)^k\zeta(\chi_{k/2},u j/3)\right)}{2\chi_{k/2}(1+\chi_{k/2})\log
    p}N^{1+\chi_{k/2}}\\+\sum_{j=-1}^1\frac{\zeta(\chi_{k/2},j/3)+(-1)^k\zeta(\chi_{k/2},u j/3)}{2\chi_{k/2}(1+\chi_{k/2})\log
    p}\biggl(\Bigl(N+\frac j3\Bigr)^{1+\chi_{k/2}}-N^{1+\chi_{k/2}}\biggr).
\end{multline}
As $\zeta(s,-1/3)+\zeta(s,0)+\zeta(s,1/3)=3^s\zeta(s)$, the first summand equals
\begin{equation*}
  \frac{1+(-1)^k}2\frac{3^{\chi_{k/2}}\zeta(\chi_{k/2})}{\chi_{k/2}(1+\chi_{k/2})\log
    p}N^{1+\chi_{k/2}},
\end{equation*}
which leads to no contribution for odd $k$. Summing over all $k\neq 0$ yields
\begin{equation*}
  \sum_{k\neq 0}r_{\chi_{k/2}}=N\Phi^{(2)}(\log_pN)+\Psi(N),
\end{equation*}
where the $1$-periodic function $\Phi^{(2)}$ is  defined in
\eqref{eq:Phi-2-definition}---it
corresponds to the first summand in \eqref{eq:split-Fourier-contribution}---and
$\Psi(N)$ is the function defined in Theorems~\ref{theorem:result-case-1} and
\ref{theorem:result-case-2}.

As 
\begin{equation*}
  \Bigl(N+\frac j3\Bigr)^{1+\chi_{k/2}}-N^{1+\chi_{k/2}}=O(\min\{N, |k|\})
\end{equation*}
and 
\begin{equation*}
  \frac{\zeta(\chi_{k/2},j/3)+(-1)^k\zeta(\chi_{k/2},u j/3)}{2\chi_{k/2}(1+\chi_{k/2})\log
    p} =O(|k|^{-3/2}\log|k|)
\end{equation*}
by \eqref{eq:growth-estimates-zeta}, we obtain
\begin{align*}
  \Psi(N)&\ll \sum_{1\le |k|\le N}|k|^{-1/2}\log|k|+\sum_{|k|>N}N|k|^{-3/2}\log|k|\\
  &\ll N^{1/2}\log N +N^{1/2}\log N\ll N^{1/2}\log N,
\end{align*}
as claimed in Theorems~\ref{theorem:result-case-1} and
\ref{theorem:result-case-2}.

As all involved Fourier series converge absolutely and uniformly, the periodic
functions are continuous.

\end{proof}

\section{Remainder Integral}
Finally, we deal with the remainder integral in \eqref{eq:residues}.
It is clear that $R=O(N^{3/4})$, but even better, it can be
computed explicitly. 

\begin{lemma}\label{lemma:remainder}We have
\begin{equation*}
  R:=\frac1{2\pi i}\sum_{j=-1}^1\int_{-\frac14-i\infty}^{-\frac14+i\infty}\Lambda_j(s)\Bigl(N+\frac
  j3\Bigr)^{s+1}\frac{ds}{s(s+1)}=\frac{2 p}{9 (p-1)}.
\end{equation*}
\end{lemma}

\begin{proof}
Expanding the denominator $p^{2s}-1$ in a geometric series
yields
\begin{equation*}
  R=-\frac{1}{2\pi i}\sum_{j=-1}^1 \sum_{\ell\ge
    0}\int_{-\frac14-i\infty}^{-\frac 14+i\infty} \Bigl(\zeta\Bigl(s,\frac
  j3\Bigr)+p^s\zeta\Bigl(s,\frac {u j}3\Bigr)\Bigr) p^{2\ell s}\Bigl(N+\frac j3\Bigr)^{s+1}\frac{ds}{s(s+1)},
\end{equation*}
where exchanging summation and integration was legitimate due to absolute
convergence.

Shifting the line of integration back to $\Re s=2$ and taking the residues in
the simple poles at $s=0$ and $s=1$ into account yields
\begin{multline*}
  R=\sum_{\ell\ge0}\biggl( -N +\frac{u+1}9+(p+1)p^{2\ell}\Bigl(\frac32N^2+\frac19\Bigr)\\-
\sum_{j=-1}^1 
\frac{1}{2\pi i}\int_{2-i\infty}^{2+i\infty} \Bigl(\zeta\Bigl(s,\frac
  j3\Bigr)+p^s\zeta\Bigl(s,\frac {u j}3\Bigr)\Bigr) p^{2\ell s}\Bigl(N+\frac j3\Bigr)^{s+1}\frac{ds}{s(s+1)}
\biggr). 
\end{multline*}
We have
\begin{multline*}
  \frac1{2\pi
    i}\int_{2-i\infty}^{2+i\infty}\zeta\Bigl(s,\frac{v j}3\Bigr)p^{ms}\Bigl(N+\frac
  j3\Bigr)^{s+1}\frac{ds}{s(s+1)}\\
  \begin{aligned}
    &=\frac{p^{-m}}{2\pi
    i}\int_{2-i\infty}^{2+i\infty}\zeta\Bigl(s,\frac{v j}3\Bigr)\Bigl(p^m N+\frac
  {(p^m-v)j}3+\frac {v j}3\Bigr)^{s+1}\frac{ds}{s(s+1)}\\
  &=p^{-m}\sum_{-v j/3<n<p^m N+(p^m-v)j/3}\Bigl(p^m N+\frac{(p^m-v)j}3-n\Bigr)
\end{aligned}
\end{multline*}
by the Mellin-Perron summation formula again. In the cases of interest, we have
either $v=u$ and $m$ odd or $v=1$ and $m$ even, which implies that $p^m-v$ is
divisible by $3$ in our cases. Thus we obtain
\begin{multline*}
    \frac1{2\pi
    i}\int_{2-i\infty}^{2+i\infty}\zeta\Bigl(s,\frac{v j}3\Bigr)p^{ms}\Bigl(N+\frac
  j3\Bigr)^{s+1}\frac{ds}{s(s+1)}\\
  =p^{-m}\frac{(p^m N +(p^m-v)j/3+\ceiling{v j/3}-1)(p^m N +(p^m-v)j/3+\ceiling{v j/3})}{2}.
\end{multline*}
Summing over $j$ yields
\begin{equation*}
  R=\sum_{\ell\ge 0}\frac{2}{9} p^{-2 \ell -1} (p+1)=\frac{2 p}{9 (p-1)}.
\end{equation*}
\end{proof}

Combining \eqref{simplified}, Lemma~\ref{lemma:Mellin-Perron-result},
Lemma~\ref{lemma:residues} and \ref{lemma:remainder} yields the results given
in Theorems~\ref{theorem:result-case-1} and \ref{theorem:result-case-2}.

\bibliographystyle{amsplain}
\bibliography{asm-divisibility-analysis}

\end{document}